\documentclass[notheoremnums]{dpreprint}

\newtheorem{theorem}{Theorem}[section]
\newtheorem*{theorem*}{Theorem}
\newtheorem{corollary}[theorem]{Corollary}
\newtheorem{proposition}[theorem]{Proposition}
\newtheorem{lemma}[theorem]{Lemma}
\theoremstyle{definition}
\newtheorem{definition}[theorem]{Definition}

\newtheorem{remark}[theorem]{Remark}
\newtheorem{question}[theorem]{Question}

\newtheorem*{Theorem1}{Theorem \ref{T:main1}}
\newtheorem*{Theorem2}{Theorem \ref{T:evenlower}}
\newtheorem*{Theorem3}{Theorem \ref{T:main2}}
\newtheorem*{Theorem4}{Theorem \ref{T:noqlogq}}
\newtheorem*{Theorem5}{Theorem \ref{T:superlinear}}

\newcommand{\len}[1]{\| #1 \|}

\usepackage{xcolor}
\usepackage{mathtools}
\DeclarePairedDelimiter{\ceil}{\lceil}{\rceil}

\begin{document}

\dtitle[Mixing Subshifts with Low Complexity]{Measure-Theoretically Mixing Subshifts with Low Complexity}
\dauthorone[D.~Creutz]{Darren Creutz}{creutz@usna.edu}{US Naval Academy}{}
\dauthortwo[R.~Pavlov]{Ronnie Pavlov}{rpavlov@du.edu}{University of Denver}{The second author gratefully acknowledges the support of a Simons Foundation Collaboration Grant.}
\dauthorthree[S.~Rodock]{Shaun Rodock}{shaunfrodock@gmail.com}{US Navy}{}
\datewritten{\today}

\keywords{Symbolic dynamics, word complexity, strong mixing, rank-one transformations}
\subjclass{Primary: 37B10; Secondary 37A25}

\dabstract{%
We introduce a class of rank-one transformations, which we call extremely elevated staircase transformations.  We prove that they are measure-theoretically mixing and, for any $f : \mathbb{N} \to \mathbb{N}$ with $f(n)/n$ increasing and $\sum 1/f(n) < \infty$, that there exists an extremely elevated staircase with word complexity $p(n) = o(f(n))$.  This improves the previously lowest known complexity for mixing subshifts, resolving a conjecture of Ferenczi.
}

\makepreprint

\section{Introduction}

It is well-known that there exist dynamical systems in which two seemingly opposite properties can coexist: zero entropy, which implies that a system is in a sense `simple' or `deterministic,' and (measure-theoretic) strong mixing, which implies that sets become `asymptotically independent' under repeated application (the first construction of such a system is due to Girsanov \cite{girsanov}, see also \cite{rokhlin67} and \cite{pinsker60}). For the symbolically defined dynamical systems known as subshifts, the concept of word complexity provides further quantification within zero entropy; zero entropy means that word complexity function $p(n)$ grows subexponentially, but of course one can study slower growth rates as well. Many recent results treat subshifts with very low complexity (see, among others, \cite{CK1}, \cite{CK2}, \cite{CK3}, \cite{DDMP}, \cite{DOP}, and \cite{PS}), showing that they must be `simple' in various ways. In contrast, our results show that such subshifts can still be `complex' in the sense of having a strong mixing measure.

Using this framework, in \cite{ferenczi1996rank} Ferenczi described a subshift example supporting a strongly mixing invariant measure whose word complexity satisfies $\frac{p(q)}{q^2} \rightarrow 0.5$. He somewhat glibly conjectured that this was the minimal possible word complexity for such a shift, but also said that he would `wait confidently for the next counterexample.' 
Ferenczi also showed that such a subshift must have $\limsup \frac{p(q)}{q} = \infty$, i.e.~its word complexity function cannot be bounded from above by any linear function.

Ferenczi's example was the symbolic model of a so-called rank-one system. Rank-one systems are traditionally defined by a cutting and stacking procedure on an interval with Lebesgue measure, but they are measure-theoretically isomorphic to the empirical measure on a recursively defined subshift (see \cite{danilenko16}, \cite{adamsferenczipeterson17}). The rank-one examples from \cite{ferenczi1996rank} are well-studied examples called staircase transformations, originally defined by Smorodinsky and Adams, and which were proved to be measure-theoretically mixing in \cite{adams1998smorodinsky}, \cite{CreutzSilva2004} and \cite{CreutzSilva2010}. 

Somewhat surprisingly, we show that a fairly simple alteration of the traditional staircase yields rank-one systems, which we call \textbf{extremely elevated staircase transformations}, which have word complexity much lower than quadratic (though unavoidably superlinear) and whose symbolic models are measure-theoretically mixing. We prove several results about how slowly complexity can grow for such examples.

We first show that the complexity $p(q)$ can grow more slowly than any sequence whose sum of reciprocals converges.

\newcommand{\theoremonetext}{%
Let $f : \mathbb{N} \to \mathbb{N}$ be a function such that $\frac{f(q)}{q}$ is nondecreasing and $\sum \limits \frac{1}{f(q)} < \infty$.
	Then there exists a (mixing) extremely elevated staircase transformation where $\lim \frac{p(q)}{f(q)} = 0$.%
}

\begin{Theorem1}
\theoremonetext
\end{Theorem1}

This is not, however, a necessary restriction on word complexity, as we can construct some examples with even slower growth.

\newcommand{\theoremtwotext}{%
There exists a (mixing) extremely elevated staircase transformation where $\sum \limits \frac{1}{p(q)} = \infty$.%
}

\begin{Theorem2}
\theoremtwotext
\end{Theorem2}

We also prove that there exist such mixing subshifts with even lower complexity along sequences.

\newcommand{\theoremthreetext}{%
For every $\epsilon > 0$, there exists a (mixing) extremely elevated staircase transformation where 
$\liminf \frac{p(q)}{q (\log q)^\epsilon} = 0$.%
}

\begin{Theorem3}
\theoremthreetext
\end{Theorem3}

However, we then show that there is a superlinear lower bound of $q \log(q)$ for the complexity function.

\newcommand{\theoremfourtext}{%
For every extremely elevated staircase transformation,
$\limsup \frac{p(q)}{q \log q} = \infty$.%
}

\begin{Theorem4}
\theoremfourtext
\end{Theorem4}

Finally, we show that extremely elevated staircase cannot achieve linear complexity even along a sequence.

\newcommand{\theoremfivetext}{%
For every extremely elevated staircase transformation, $\lim \frac{p(q)}{q} = \infty$.%
}

\begin{Theorem5}
\theoremfivetext
\end{Theorem5}

In the spirit of Ferenczi's `waiting confidently for the next counterexample,' we also wonder whether there are other classes of subshifts supporting mixing measures which can achieve even lower complexity.

\begin{question}\label{liminf2}
Is there any nontrivial lower bound on complexity growth for all subshifts with a mixing measure, i.e., does there exist $f > 1$ so that $\liminf \frac{p(q)}{qf(q)} > 1$ for all such subshifts?
\end{question}

\begin{question}
Is there a superlinear lower bound on complexity growth along a sequence for all subshifts with a mixing measure, i.e., does there exist unbounded $g$ so that $\limsup \frac{p(q)}{qg(q)} = \infty$ for all such subshifts?
\end{question}

We note that in Question~\ref{liminf2}, we chose phrasing to admit the possibility that there exist such examples which have linear complexity along a subsequence, as this was not ruled out by Ferenczi's results and we do not know whether it is possible.

\section{Definitions and preliminaries}

\subsection{General symbolic dynamics and ergodic theory}

We begin with some general definitions in ergodic theory. 

\begin{definition}
A \textbf{measure-theoretic dynamical system} or \textbf{MDS} is a quadruple $(X, \mathcal{B}, \mu, T)$, where $(X, \mathcal{B}, \mu)$ is a standard Borel or Lebesgue measure space and $T: X \rightarrow X$ is an invertible measure-preserving map, i.e.~$\mu(T^{-1} A) = \mu(A)$ for all $A \in \mathcal{B}$.
\end{definition}

\begin{definition}
An MDS $(X, \mathcal{B}, \mu, T)$ is \textbf{ergodic} if $A = T^{-1} A$ implies that $\mu(A) = 0$ or $\mu(A^c) = 0$. 
\end{definition}

A crucial usage of ergodicity is the mean ergodic theorem:
\begin{theorem}
If $(X, \mathcal{B}, \mu, T)$ is ergodic, then for any $f \in L^2(X)$ with $\int f \ d\mu = 0$, 
\[
\lim_{n \rightarrow \infty} \int \big{|}\frac{1}{n} \sum_{i = 0}^{n-1} f \circ T^{-i} \big{|}^{2} \ d\mu = 0.
\]
\end{theorem}

\begin{definition}
An MDS $(X, \mathcal{B}, \mu, T)$ is \textbf{strongly mixing} if for all $A, B \in \mathcal{B}$, $\mu(A \cap T^{-n} B) \rightarrow \mu(A) \mu(B)$.
\end{definition}

\begin{definition}
An MDS $(X, \mathcal{B}, \mu, T)$ and an MDS $(X', \mathcal{B}', \mu', T')$ are \textbf{measure-theoretically isomorphic} if there exists
a bijective map $\phi$ between full measure subsets $X_0 \subset X$ and $X'_0 \subset X'$ where 
$\mu(\phi^{-1} A) = \mu'(A)$ for all measurable $A \subset X'_0$ and $(\phi \circ T)x = (T' \circ \phi)x$ for all $x \in X_0$.
\end{definition}

Most of the systems we study in this work will be symbolically defined systems called subshifts.

\begin{definition}
A \textbf{subshift} on the finite set $\mathcal{A}$ is any subset $X \subset \mathcal{A}^{\mathbb{Z}}$ which is closed in the product topology and shift-invariant, i.e.~for all $x = (x(n))_{n \in \mathbb{Z}} \in X$ and $k \in \mathbb{Z}$, the translation $(x(n+k))_{n \in \mathbb{Z}}$ of $x$ by $k$ is also in $X$.
\end{definition} 

\begin{definition}
A \textbf{word} on the finite set $\mathcal{A}$ is any element of $\mathcal{A}^n$ for some $n$, which is called the \textbf{length} of $w$ which we denote $\len{w}$. A word $w$ of length $\ell$ is said to be a \textbf{subword} of a word or biinfinite sequence $x$ if there exists $k$ so that $w(i) = x(i+k)$ for all $1 \leq i \leq \ell$. When $x$ is a word, say with length $m$, we say that $w$ is a \textbf{prefix} of $x$ if it occurs at the beginning of $x$ (i.e.~$k = 0$ in the above) and a \textbf{suffix} of $x$ if it occurs at the end of $x$ (i.e.~$k = m - \ell$ in the above). 
\end{definition}

For words $v,w$, we denote by $vw$ their concatenation, i.e.~the word obtained by following $v$ immediately by $w$. We use similar notation for concatenations of multiple words, e.g., $w_1 w_2 \ldots w_n$. When it is notationally convenient, we may sometimes refer to such a concatenation with product or exponential notation, e.g., $\prod_i w_i$ or $0^n$.

\begin{definition}
The \textbf{language} of a subshift $X$, denoted $\mathcal{L}(X)$, is the set of all words $w$ which are subwords of some $x \in X$. 
\end{definition}

\begin{definition}
The \textbf{word complexity function} of a subshift $X$ over $\mathcal{A}$ is the function $p_X: \mathbb{N} \rightarrow \mathbb{N}$ defined by $p_X(n) = |\mathcal{L}(X) \cap \mathcal{A}^n|$, the number of words of length $n$ in the language of $X$.
\end{definition}

When $X$ is clear from context, we suppress the subscript and just write $p(n)$.

\begin{definition}
A word $w$ is \textbf{right-special} in a subshift $X$ over $\{0,1\}$ if $w0, w1 \in \mathcal{L}(X)$.
\end{definition}

We note that this property is often called \textbf{right special} in the literature.
All subshifts we examine are on the alphabet $\{0,1\}$, and in this setting we will repeatedly make use of the following basic lemma due to Cassaigne \cite{cassaigne}.

\begin{lemma}
For any subshift $X$ over $\{0,1\}$, if we denote by $\mathcal{L}^{RS}_\ell(X)$ the set of right-special words in $X$ of length $\ell$, then for all positive $m < n$,
\[
p(n) = p(m) + \sum_{\ell = m}^{n-1} |\mathcal{L}^{RS}_\ell(X)|.
\]
\end{lemma}

The classical Hedlund-Morse theorem (\cite{morse1938symbolic}) states that every infinite subshift $X$ has at least one right-special word for each length, and so every such subshift satisfies $p(n) > n$ for all $n$.

\subsection{Rank-one transformations and their symbolic models}

A \textbf{rank-one transformation} is an MDS $(X, \mathcal{B}(X), m, T)$ (from now on referred to just as $(X,T)$) constructed by a so-called cutting and stacking construction; here $X$ represents a (possibly infinite) interval, $\mathcal{B}(X)$ is the induced Borel $\sigma$-algebra from $\mathbb{R}$, and $m$ is Lebesgue measure. We give only a brief introduction here, and refer the reader to 
\cite{fghsw21} or  \cite{silva2008invitation} for a more detailed presentation.

The transformation $T$ is defined inductively on larger and larger portions of the space by the use of Rokhlin towers or \textbf{columns}, denoted $C_n$. Each column $C_n$ consists of \textbf{levels} $I_{n,a}$ where $0 \leq a < h_{n}$ is the height of the level within the column. All levels $I_{n,a}$ in $C_n$ are 
intervals with the same length, and the total number of levels in a column is the \textbf{height} of the column, denoted by $h_n$. The transformation $T$ is defined on all levels $I_{n,a}$ except the top one $I_{n, h_n - 1}$ by sending each $I_{n,a}$ to $I_{n,a+1}$ using the unique affine map between them.

We start with $C_1=[0,1)$ with height $h_1=1$. To obtain $C_{n+1}$ from $C_n$, we require a \textbf{cut sequence}, $\{r_n\}$ such that $r_n \geq 1 \: \forall n$. For each $n$, we make $r_n$ vertical cuts of $C_n$ to create $r_n+1$ \textbf{subcolumns} of equal width. We denote a \textbf{sublevel} of $C_n$ by $I_{n,a}^{[i]}$ where $0 \leq a < h_{n}$ is the height of the level within that column, and 
$i$ represents the position of the subcolumn, where $i=0$ represents the leftmost subcolumn and $i=r_n$ is the rightmost subcolumn. After cutting $C_n$ into subcolumns, we add extra intervals called \textbf{spacers} on top of each subcolumn to function as levels of the next column. The \textbf{spacer sequence}, $\{s_{n,i}\}$, specifies how many sublevels to add above each subcolumn where $n$ represents the column we are working with, $i$ represents the subcolumn that spacers are added above, and $s_{n,i} \geq 0$ for $0 \leq i\leq r_n$. Spacers are the same width as the sublevels, act as new levels in the column $C_{n+1}$, and are always taken to be the leftmost intervals in $\mathbb{R}$ not currently part of a level.
Once the spacers are added on top of the subcolumns, we stack the subcolumns with their spacers right on top of left. This gives us the next column, $C_{n+1}$.

Each column $C_n$ yields a definition of $T$ on $\bigcup_{a = 0}^{h_n - 2} I_{n,a}$; it is routine to check that 
the partially defined map $T$ on $C_{n+1}$ agrees with that of $C_n$, extending the definition of $T$ to a portion of the top level of $C_n$, where it was previously undefined. Continuing this process gives the \textbf{sequence of columns} $\{C_1, \dots, C_n, C_{n+1}, \dots\}$ and $T$ is then the limit of the partially defined maps. 

Though in theory this construction could result in $X$ being an infinite interval with infinite Lebesgue measure, it is known that $X$ has finite measure if and only if 
$\sum_{n} \frac{1}{r_{n}h_{n}}\sum_{i=0}^{r_{n}} s_{n,i} < \infty$ (see e.g.~\cite{CreutzSilva2010}).
All rank-one transformations we define will satisfy this condition, and for convenience we always renormalize so that $X = [0,1)$. Since $X$ is always $[0,1)$ equipped with the Lebesgue measure, we hereafter refer to the MDS by just the map $T$. Every rank-one transformation $T$ is an invertible and ergodic MDS.

\begin{remark}
The reader should be aware that we are making $r_{n}$ cuts and obtaining $r_{n}+1$ subcolumns (following Ferenczi \cite{ferenczi1996rank}), while other papers (e.g.~\cite{Creutz2021}) use $r_{n}$ as the number of subcolumns.
\end{remark}

We will later need the following general bounds for rank-one transformations.

\begin{proposition}\label{P:hr}
	Let $\{ r_{n} \}$ and $\{ h_{n} \}$ be the cut and height sequences for a rank-one transformation on a probability space with initial base level $C_{1}$.
	Then
	\[
	\prod_{j=1}^{n-1} (r_{j}+1) \leq h_{n} \leq \frac{1}{\mu(C_{1})}\prod_{j=1}^{n-1} (r_{j}+1) \quad\quad\quad\quad\text{and}\quad\quad\quad\quad \frac{1}{h_{n}}\prod \limits_{j=1}^{n-1} (r_{j}+1) \to \mu(C_{1}).
	\]
\end{proposition}

\begin{proof}
	Define
	$s_{n} = \frac{1}{r_{n}+1}\sum \limits_{i=0}^{r_{n}} s_{n,i}$ where $\{ s_{n,i} \}_{\{r_{n}\}}$ is the spacer sequence so
	$\mu(C_{n+1}) = \mu(C_{n}) + s_{n}\mu(I_{n}) = \mu(C_{n})\big{(}1 + \frac{s_{n}}{h_{n}}\big{)}$, meaning
	$\mu(C_{n}) = \mu(C_{1})~\prod \limits_{j=1}^{n-1}(1 + \frac{s_{j}}{h_{j}})$.
	Since
	$h_{n+1} = (r_{n}+1)h_{n} + \sum \limits_{i=0}^{r_{n}}s_{n,i} = (r_{n}+1)h_{n} \big{(} 1 + \frac{s_{n}}{h_{n}} \big{)}$
	and $h_{0} = 1$, we have $h_{n} = \prod \limits_{j=1}^{n-1} (r_{j}+1)(1 + \frac{s_{j}}{h_{j}}) = \left(\prod \limits_{j=1}^{n-1}(r_{j}+1)\right)\frac{\mu(C_{n})}{\mu(C_{1})}$ and $\mu(C_{n}) \to 1$.
\end{proof}

In order to discuss word complexity for rank-one transformations, we need to deal with symbolic models. Suppose that $T$ is a rank-one system as defined above, with associated $\{ r_n \}$ and $\{ s_{n,i} \}$. We will define a subshift $X(T)$ with alphabet $\{0,1\}$ which is measure-theoretically isomorphic to $T$. Define a sequence of words as follows: $B_1 = 0$, and for every 
$n > 1$, 
\begin{equation*}\label{rk1word}
B_{n+1}=B_n1^{s_{n,0}}B_n1^{s_{n,1}}\dots 1^{s_{n,r_n}}= \prod_{i=0}^{r_n} B_n1^{s_{n,i}}.
\end{equation*}

The motivation here should be clear; $B_n$ is a symbolic coding of the column $C_n$, where $0$ represents levels which come from the first column $C_1$, and $1$ represents levels which are spacers. Define $X(T)$ to consist of all biinfinite $\{0,1\}$ sequences where every subword is a subword of some $B_n$. We note that $X(T)$ is not uniquely ergodic if the spacer sequence $\{ s_{n,i} \}$ is unbounded (which will always be the case for us), since the sequence $1^{\infty}$ is always in $X(T)$. Nevertheless, there is a `natural' measure associated to $X(T)$:

\begin{definition}
The \textbf{empirical measure} for a symbolic model $X(T)$ of a rank-one system $T$ is the measure $\mu$ defined by
\[
\mu([w]) := \lim_{n \rightarrow \infty} \frac{|\{i \ : \ B_n(i) \ldots B_n(i + \ell - 1) = w\}|}{|B_n|}
\]
for every $\ell$ and every word $w$ of length $\ell$.
\end{definition}

It was proved in \cite{danilenko16}, \cite{adamsferenczipeterson17} (see \cite{fghsw21} for a more general definition of rank-one which includes odometers in the symbolic setting) that a rank-one MDS $T$ and its symbolic model $X(T)$ (with empirical measure $\mu$) are always measure-theoretically isomorphic, and so the symbolic model is measure-theoretically mixing iff the original rank-one was. Due to this isomorphism, in the sequel we move back and forth between rank-one and symbolic model terminology as needed. For simplicity, we from now on write $\mathcal{L}(T)$ for the language of $X(T)$, and define:

\begin{definition}
A \textbf{mixing rank-one subshift} is a symbolic model of a rank-one transformation that is mixing with respect to its empirical measure.
\end{definition}

\section{Extremely elevated staircase transformations}

\begin{definition}  An \textbf{extremely elevated staircase transformation} is a rank-one transformation defined by cut sequence $\{ r_{n} \}$ and \textbf{elevating sequence} $\{ c_{n} \}$ with spacer sequence given by $s_{n,j} = c_{n} + i$ for $0 \leq i < r_{n}$ and $s_{n,r_{n}} = 0$.
The cut sequence $\{ r_{n} \}$ is required to be nondecreasing to infinity with $\frac{r_{n}^{2}}{h_{n}} \to 0$
and the elevating sequence $\{ c_{n} \}$ to satisfy $c_{1} \geq 1$ and $c_{n+1} \geq h_{n} + 2c_{n} + 2r_{n} - 2$ and $\sum \frac{c_{n}+r_{n}}{h_{n}} < \infty$.
\end{definition}

\begin{theorem}\label{T:mixing}
	Let $T$ be an extremely elevated staircase transformation.  Then $T$ is mixing (on a finite measure space).
\end{theorem}

The proof of Theorem \ref{T:mixing} is postponed to the appendix.

The symbolic representation of an extremely elevated staircase is $B_{1} = 0$ and $h_{1} = 1$ and,
\[
B_{n+1} = \Big{(}\prod_{i=0}^{r_{n}-1}B_{n}1^{c_{n}+i}\Big{)}B_{n} \quad\quad\text{and}\quad\quad
h_{n+1} = (r_{n}+1)h_{n} + r_{n}c_{n} + \frac{1}{2}r_{n}(r_{n}-1).
\]

\subsection{Right-special words in the language of \texorpdfstring{$T$}{T}}

\begin{proposition}\label{P:twosuccs}
Let $T$ be an extremely elevated staircase transformation with language $\mathcal{L}(T)$.
If $w \in \mathcal{L}(T)$ is right-special then exactly one of the following holds:
\begin{enumerate}[\hspace{10pt}(i)\hspace{5pt}]
\item\label{Pt-1} $w = 1^{\len{w}}$; or
\item\label{Pt-2} $w$ is a suffix of $1^{c_{n}+r_{n}-1}B_{n}1^{c_{n}}$ for some $n$ and $\len{w} > c_{n}$; or
\item\label{Pt-3} $w$ is a suffix of $1^{c_{n}+i-1}B_{n}1^{c_{n}+i}$ for some $n$ and $0 < i < r_{n}$ and $\len{w} > c_{n}+i$.
\end{enumerate}
\end{proposition}
\begin{proof}
If $01^{t}0 \in \mathcal{L}(T)$ then there exists $m \geq 1$ and $0 \leq j < r_{m}$ such that $t = c_{m} + j$ as only spacer sequences can appear between $0$s.  Since $c_{n+1} \geq c_{n} + r_{n}$, for any such word the choice of $m$ is unique.  Moreover, since $01^{c_{m}+j}0$ only appears in $B_{m+1}$, which is always preceded by $1^{c_{m+1}}$, the word $01^{c_{m}+j}0$ only appears as a suffix of $1^{c_{m+1}}(\prod_{k=0}^{j}B_{m}1^{c_{m}+k})0$.

Let $w \in \mathcal{L}(T)$ be a right-special word.  Since $c_{1} \geq 1$, the word $00 \notin \mathcal{L}(T)$ so $w$ does not end with $0$.  If $w = 1^{\len{w}}$, it is of form $(\ref{Pt-1})$.  So we may assume that $w$ ends with $1$ and contains at least one $0$.

Let $z \in \mathbb{N}$ such that $w$ has $01^{z}$ as a suffix.

Since $w0 \in \mathcal{L}(T)$, $01^{z}0 \in \mathcal{L}(T)$ so there exists a unique $n \geq 1$ and $0 \leq i < r_{n}$ such that $z = c_{n} + i$.

First consider when $i > 0$.  
The word $w0$ has $01^{c_{n}+i}0$ as a suffix and that word only appears in the word $B_{n+1}$ meaning that $w0$ and $1^{c_{n+1}}(\prod_{j=0}^{i}B_{n}1^{c_{n}+j})0$ have a common suffix.

If $w$ has $01^{c_{n}+i-1}B_{n}1^{c_{n}+i}$ as a suffix then $w1$ has $01^{c_{n}+i-1}B_{n}1^{c_{n}+i+1}$ as a suffix but $01^{c_{n}+i-1}B_{n}1^{c_{n}+i+1} \notin \mathcal{L}(T)$.  Therefore $w$ is a suffix of $1^{c_{n}+i-1}B_{n}1^{c_{n}+i}$ and has length $\len{w} \geq c_{n} + i + 1$ so $w$ is of form $(\ref{Pt-3})$.

We are left with the case when $i = 0$, i.e.~when $z = c_{n}$.

The word $w0$ has $01^{c_{n}}0$ as a suffix and $01^{c_{n}}0$ only appears in the word $B_{n+1}$, and only immediately after the first $B_{n}$ in $B_{n+1}$.  As  the word $B_{n+1}$ is always preceded by $1^{c_{n+1}}$, then $w0$ and $1^{c_{n+1}}B_{n}1^{c_{n}}0$ have a common suffix.

If $w$ has $1^{c_{n}+r_{n}}B_{n}1^{c_{n}}$ as a suffix then $w1$ has $1^{c_{n} + r_{n}}B_{n}1^{c_{n}+1}$ as a suffix but $1^{c_{n} + r_{n}}B_{n}1^{c_{n}+1} \notin \mathcal{L}(T)$.

So $w$ is a suffix of $1^{c_{n}+r_{n}-1}B_{n}1^{c_{n}}$ of length $\len{w} \geq c_{n} + 1$ meaning $w$ is of form $(\ref{Pt-2})$.
\end{proof}

\begin{lemma}\label{twosucc1stcase}
	$1^{\ell}$ is right-special for all $\ell$.
\end{lemma}

\begin{proof}
	Find $n$ such that $\ell \leq \len{1^{c_{n}}}$.  Then $1^{\ell}0$ is a suffix of $1^{c_{n}}0$ and $1^{\ell}1$ is a suffix of $1^{c_{n}+1}$.
\end{proof}

\begin{lemma} \label{twosucc2ndcase}
	If $w$ is a suffix of $1^{c_{n}+r_{n}-1}B_{n}1^{c_{n}}$ then $w$ is right-special.
\end{lemma}

\begin{proof}
Choose any such $w$.  Observe that $B_{n+2}$ has $B_{n+1}1^{c_{n+1}}B_{n+1}$ as a subword and that has the subword $B_{n+1}1^{c_{n+1}}B_{n}1^{c_{n}}B_{n}$.  That word has $1^{c_{n}+r_{n}-1}B_{n}1^{c_{n}}0$ as a subword since $c_{n}+r_{n}-1 < c_{n+1}$ and so $w0$, being a suffix of $1^{c_{n}+r_{n}-1}B_{n}1^{c_{n}}0$, is in $\mathcal{L}(T)$.  Also $B_{n+2}$ has $B_{n+1}1^{c_{n+1}}$ as a subword which has  $1^{c_{n}+r_{n}-1}B_{n}1^{c_{n+1}}$ as a subword which then has $1^{c_{n}+r_{n}-1}B_{n}1^{c_{n}}1$ as a subword.  As $w1$ is a suffix of that word, $w1 \in \mathcal{L}(T)$.
\end{proof}

\begin{lemma}\label{twosucc3rdcase}
	If $w$ is a suffix of $1^{c_{n}+i-1}B_{n}1^{c_{n}+i}$ for $0 < i < r_{n}$ then $w$ is right-special.
\end{lemma}

\begin{proof}  
Choose any such $w$.
 Since$B_{n+1}$ has $1^{c_{n}+i-1}B_{n}1^{c_{n}+i}B_{n}$ as a subword, $1^{c_{n}+i-1}B_{n}1^{c_{n}+i}0 \in \mathcal{L}(T)$.  When $i < r_{n} - 1$, $B_{n+1}$ has $1^{c_{n}+i}B_{n}1^{c_{n}+i+1}$ as a subword which gives $11^{c_{n}+i-1}B_{n}1^{c_{n}+i}1$; when $i = r_{n}-1$, $B_{n+2}$ has $1^{c_{n}+r_{n}-1}B_{n}1^{c_{n+1}}$ as a subword which gives $11^{c_{n} + r_{n} - 2}B_{n}1^{c_{n}+r_{n}-1}1$ as $r_{n} < c_{n+1}$.   As $w$ is a suffix of $1^{c_{n}+i-1}B_{n}1^{c_{n}+i}$, it is right-special.
\end{proof}

\begin{lemma}\label{L:2s}
Let $T$ be an extremely elevated staircase transformation.
	For $w \in \mathcal{L}(T)$, let $n$ be the unique integer such that $w$ has $1^{c_{n}}$ as a subword and does not have $1^{c_{n+1}}$ as a subword.
	
	Then $w$ is right-special if and only if exactly one of the following holds:
\begin{enumerate}[\hspace{10pt}$(i)_{n}$\hspace{5pt}]
\item\label{Lt-0} $w = 1^{\len{w}}$ and $c_{n} \leq \ell < c_{n+1}$; or
\item\label{Lt-2} $w$ is a suffix of $1^{c_{n}+i-1}B_{n}1^{c_{n}+i}$ and $\len{w} > c_{n}+i$ for some $0 \leq i < r_{n}$; or
\item\label{Lt-3} $w$ is a suffix of $1^{c_{n}+r_{n}-1}B_{n}1^{c_{n}}$ and $\len{w} \geq h_{n} + 2c_{n}$.
\end{enumerate}
\end{lemma}
\begin{proof}
	The only words in Proposition \ref{P:twosuccs} which have $1^{c_{n}}$ as a subword, $1^{c_{n+1}}$ not a subword and at least one $0$ are of the stated forms and Lemmas \ref{twosucc1stcase}, \ref{twosucc2ndcase} and \ref{twosucc3rdcase} state that these words are right-special.  The restriction on $len{w}$ in form $(\ref{Lt-3})_{n}$ prevents any overlap between forms $(\ref{Lt-2})_{n}$ and $(\ref{Lt-3})_{n}$; the requirement that $len{w} > c_{n} + i$ ensures no overlap with form $(\ref{Lt-0})_{n}$ by either of the other two.
\end{proof}

The largest length we need consider for a given $n$ is then $h_{n} + 2c_{n} + 2(r_{n}-1) - 1$, explaining the requirement on $c_{n+1}$ in the definition of extremely elevated staircases and leading to:
\begin{definition}
The \textbf{post-productive sequence}  is
$
m_{n} = h_{n} + 2c_{n} + 2r_{n} - 2
$.
\end{definition}

\begin{proposition}\label{thing}
For an extremely elevated staircase transformation,
there is at most one right-special word of each of the forms in Lemma \ref{L:2s} and
\begin{enumerate}[\hspace{10pt}$(i)_{n}$\hspace{5pt}]
\item\label{p1} there is a word of form $(\ref{Lt-0})_{n}$ only for $c_{n} \leq \ell < c_{n+1}$; and
\item\label{p2} for each $0 \leq i < r_{n}$, there is a word of form $(\ref{Lt-2})_{n}$ for that value of $i$ only for $c_{n} + i < \ell \leq h_{n} + 2c_{n} + 2i - 1$; and
\item\label{pspec} there is a word of form $(\ref{Lt-3})_{n}$ only for $h_{n} + 2c_{n} \leq \ell < h_{n} + 2c_{n} + r_{n}$.
\end{enumerate}
\end{proposition}
\begin{proof}
Every $w$ of a form in Lemma \ref{L:2s} for a given $n$ has length $c_{n} \leq len{w} < m_{n} \leq c_{n+1}$ so for every length $\ell$ there is exactly one $n$ for which Lemma \ref{L:2s} could potentially give a right-special word.

$1^{\ell}$ is of form $(\ref{Lt-0})_{n}$ for $c_{n} \leq \ell < c_{n+1}$.

If $w$ is of form $(\ref{Lt-2})_{n}$, itis a suffix of $1^{c_{n}+r_{n}-1}B_{n}1^{c_{n}}$ so $\len{w} \leq \len{1^{c_{n}+r_{n}-1}B_{n}1^{c_{n}}} = h_{n} + 2c_{n} + r_{n} - 1$.

If $w$ is of form $(\ref{Lt-3})_{n}$, it is a suffix of $1^{c_{n}+i-1}B_{n}1^{c_{n}+i}$ so $\len{w} \leq \len{1^{c_{n}+i-1}B_{n}1^{c_{n}+i}} = h_{n} + 2c_{n} + 2i - 1$.
\end{proof}

\subsection{Counting right-special words of length \texorpdfstring{$\ell$}{l} for extremely elevated staircases}

\begin{lemma}\label{cf1}
	If $c_n \leq \ell < c_n + r_n$ then $p(\ell +1)-p(\ell)=(\ell-c_n)+1$.
\end{lemma}

\begin{proof}
	Proposition \ref{thing} gives one word of form $(\ref{p1})_{n}$ and one of form $(\ref{p2})_{n}$ for each $0 \leq i < \ell - c_{n}$.
\end{proof}

\begin{lemma} \label{cf2}
	If $c_n+r_n \leq \ell \leq h_n+2c_n+1$ then $p(\ell+1)-p(\ell)=r_n+1$.
\end{lemma}

\begin{proof}
	Proposition \ref{thing} gives one word of form $(\ref{p1})_{n}$ and one for each $0 \leq i < r_{n}$ of form $(\ref{p2})_{n}$.
\end{proof}

\begin{lemma}\label{cf3.1}
	If $h_n+2c_n+1 < \ell \leq h_n+2c_n+r_n-1$ then $p(\ell + 1) - p(\ell) = r_n - \lceil \frac{1}{2}(\ell-(h_n+2c_n+1)) \rceil + 1$.
\end{lemma}

\begin{proof}
	Proposition \ref{thing} gives one word of form $(\ref{p1})_{n}$, one word of form $(\ref{pspec})_{n}$ and, for $0 \leq i < r_{n}$, one of form $(\ref{p2})$ for $0 \leq i < r_{n}$ only if $\ell \leq h_{n} + 2c_{n} + 2i - 1$ so only when $x = \ell - h_{n} - 2c_{n} - 1 \leq 2i - 2$ so only when $i \geq \lceil (x+2)/2 \rceil$.  This gives exactly $r_{n}-1 - \lceil x/2 \rceil$ words of form $(\ref{p2})_{n}$.
\end{proof}

\begin{lemma}\label{cf3.2}
	If $h_n+2c_n+r_n \leq \ell \leq h_n+2c_n+2r_n-3$ then $p(\ell + 1) - p(\ell) = r_n - \lceil\frac{1}{2}(\ell-(h_n+2c_n+1)) \rceil$.
\end{lemma}

\begin{proof}
	The proof of Lemma \ref{cf3.1} holds here except we do not get a word of form $(\ref{pspec})_{n}$.
\end{proof}

\begin{lemma}\label{cf4}
	If $m_{n} \leq \ell < c_{n+1}$, then $p(\ell+1)-p(\ell)=1$.
\end{lemma}

\begin{proof}
	 Proposition \ref{thing} gives only the word $1^{\ell}$ of length $\ell \geq m_{n}$.
\end{proof}

\subsection{Counting words in the language of extremely elevated staircases}

\begin{proposition} \label{C:all} 
	If $T$ is an extremely elevated staircase transformation and $c_n<q\leq c_{n+1}$, then
	\[
	p(q) \leq p(c_{n}) + (q - c_{n})(r_{n}+1) \leq q(r_n+1).
	\]
\end{proposition}

\begin{proof} 
	From Lemmas \ref{cf1}--\ref{cf4}, for $c_m\leq\ell < c_{m+1}$ it always holds that $p(\ell+1)-p(\ell) \leq r_m+1$ so
	\[
	p(q) = p(c_{n}) + \sum_{\ell=c_{n}}^{q-1} (p(\ell+1)-p(\ell)) \leq p(c_{n}) + (q - c_{n})(r_{n} + 1)
	\]
and, since $r_m \leq r_n$ for all $m\leq n$,,
	\[
	p(c_{n}) = \sum \limits_{\ell=1}^{c_{n}} \left(p(\ell+1)-p(\ell)\right) \leq \sum \limits_{\ell=1}^{c_{n}} (r_n+1) = c_{n}(r_n+1). \qedhere
	\]
\end{proof}

\begin{proposition}\label{P:lowerbound}
For an extremely elevated staircase transformation,
$
p(m_{n}) \geq h_{n+1}.
$
\end{proposition}
\begin{proof}
By Lemma \ref{cf1}, $p(c_{n}+r_{n}) - p(c_{n}) = \frac{1}{2}r_{n}(r_{n}+1)$.  There are $r_{n}-2+\sum_{x=0}^{2(r_{n}-2)} (r_{n} - \lceil\frac{x}{2}\rceil)$ words from Lemmas \ref{cf3.1} and \ref{cf3.2} of lengths $h_{n} + 2_{n} + 2 \leq \ell \leq h_{n} + 2c_{n} + 2r_{n} - 3$, therefore $p(h_{n}+2c_{n}+2r_{n}-2) - p(h_{n}+2c_{n}+1) = r_{n}^{2}-4$.  By Lemma \ref{cf2}, $p(h_{n}+2c_{n}+1) - p(c_{n}+r_{n}) = (r_{n}+1)(h_{n}+c_{n}-r_{n}+2)$ so
\[
p(h_{n} + 2c_{n}+2r_{n}-2) \geq \frac{1}{2}r_{n}(r_{n}+1) + (r_{n}+1)(h_{n}+c_{n}-r_{n}+2) + r_{n}^{2} - 4
\geq h_{n+1}. \qedhere
\]
\end{proof}

\section{Mixing rank-one subshifts with low complexity}

\begin{theorem}\label{T:main1}
\theoremonetext
\end{theorem}

\begin{proof}
	The function $g(q) = \min(f(q),q^{3/2})$ is nondecreasing as it is the minimum of two nondecreasing functions and $\frac{g(q)}{q}$ is the minimum of $\frac{f(q)}{q}$ and $q^{1/2}$ so is also nondecreasing.  Replacing $f(q)$ by $g(q)$ if necessary, we may assume that $f(q) \leq q^{3/2}$ for all $q$.
	
	Note that $\frac{f(q)}{q} \to \infty$ since it is nondecreasing and if $f(q) \leq Cq$ then $\sum \frac{1}{f(q)} \geq (1/C) \sum \frac{1}{q} = \infty$.
	
	Set $x_{1} = 1$ and choose $x_{t}$ such that $\sum_{q=x_{t}}^{\infty} \frac{1}{f(q)} \leq t^{-3}$ and$\frac{f(q)}{q} \geq t^{2}$ for $q \geq x_{t}$.

	Set $r_{1} = 2$ and $c_{1} = 1$.  Given $r_{n}$ and $c_{n}$, let $t_{n}$ such that $x_{t_{n}} \leq c_{n} < x_{t_{n}+1}$ and	
	set
	\[
	c_{n+1} = m_{n} \text{ and }r_{n+1} = \ceil[\Big]{\frac{f(c_{n+1})}{t_{n}(c_{n+1} - c_{n})}}.
	\]
	Since $r_{n+1} \geq \frac{f(c_{n+1})}{c_{n+1}}\cdot \frac{1}{t_{n}} \geq \frac{t_{n}^{2}}{t_{n}} \to \infty$, we have that $r_{n}$ nondecreasing to $\infty$.  
	
	Let $n_{t} = \inf \{ n : c_{n} \geq x_{t} \}$ so that $t_{n} = t$ for $n_{t} \leq n < n_{t+1}$.
	Since $f$ is increasing,
	\begin{align*}
	\sum_{n=1}^{\infty} \frac{1}{r_{n}} &\leq \sum_{n=1}^{\infty} \frac{1}{\frac{f(c_{n})}{t_{n-1}(c_{n} - c_{n-1})}} = \sum_{n=1}^{\infty} \frac{t_{n-1}(c_{n} - c_{n-1})}{f(c_{n})}= \sum_{n=1}^{\infty} \sum_{\ell=c_{n-1}}^{c_{n}-1} \frac{t_{n-1}}{f(c_{n})} \\
	&\leq \sum_{n=1}^{\infty} \sum_{\ell=c_{n-1}}^{c_{n}-1} \frac{t_{n-1}}{f(\ell)} = \sum_{t=1}^{\infty} \sum_{n=n_{t}+1}^{n_{t+1}} \sum_{\ell=c_{n-1}}^{c_{n}+1} \frac{t}{f(\ell)}
	= \sum_{t=1}^{\infty} \sum_{\ell=c_{n_{t}}}^{c_{n_{t+1}}-1} \frac{t}{f(\ell)} \\
	&\leq \sum_{t=1}^{\infty} t \sum_{\ell=x_{t}}^{\infty} \frac{1}{f(\ell)}
	\leq \sum_{t=1}^{\infty} \frac{t}{t^{3}} < \infty.
	\end{align*}
	
	Since $h_{n+1} \geq r_{n}(h_{n} + c_{n})$ and $2r_{n} \leq h_{n}$,
\[
\sum_{n} \frac{c_{n+1}}{h_{n+1}} \leq \sum_{n} \frac{h_{n} + 2c_{n} + 2r_{n} - 2}{r_{n}(h_{n}+c_{n})} 
\leq \sum_{n} \frac{2(h_{n}+c_{n})}{r_{n}(h_{n} + c_{n})} = 2\sum_{n} \frac{1}{r_{n}}
\]
	and therefore $\sum \frac{c_{n}}{h_{n}} < \infty$.
		Since $f(q) \leq q^{3/2}$,
	\begin{align*}
	\frac{r_{n}^{2}}{h_{n}} &\leq \frac{(f(c_{n}))^{2}}{h_{n}t_{n-1}^{2}(c_{n} - c_{n-1})^{2}} \leq \frac{(c_{n}^{3/2})^{2}}{h_{n}c_{n}^{2}}\Big{(}\frac{c_{n}}{c_{n} - c_{n-1}}\Big{)}^{2} \frac{1}{t_{n-1}^{2}} = \frac{c_{n}}{h_{n}} \Big{(}\frac{1}{1 - \frac{c_{n-1}}{c_{n}}}\Big{)}^{2} \frac{1}{t_{n-1}^{2}} \to 0.
	\end{align*}
	as $\frac{c_{n-1}}{c_{n}} \leq \frac{c_{n-1}}{h_{n-1}} \to 0$.
	Then the transformation $T$ with cut sequence $\{ r_{n} \}$ and elevating sequence $\{ c_{n} \}$ satisfies all the conditions required to be an extremely elevated staircase so Theorem \ref{T:mixing} gives that $T$ is mixing on a finite measure space.
		
	Given $q$, choose $n$ such that $c_{n} < q \leq c_{n+1}$.
	Using the fact that $\frac{f(q)}{q}$ is nondecreasing (and so $q > c_{n}$ implies $\frac{f(c_{n})}{c_{n}} \leq \frac{f(q)}{q}$) and tends to infinity, by Proposition \ref{C:all},
	\begin{align*}
	\frac{p(q)}{f(q)} &\leq \frac{q(r_{n} + 1)}{f(q)}
	\leq \frac{q}{f(q)} \Big{(}\frac{f(c_{n})}{t_{n-1}(c_{n} - c_{n-1})} + 2\Big{)}
	=  \frac{q}{f(q)} \Big{(}\frac{1}{t_{n-1}}  \frac{f(c_{n})}{c_{n}}~\frac{1}{1 - \frac{c_{n-1}}{c_{n}}} + 2 \Big{)} \\
	&\leq \frac{q}{f(q)} \Big{(} \frac{1}{t_{n-1}}\frac{f(q)}{q}~\frac{1}{1 - \frac{c_{n-1}}{c_{n}}} + 2 \Big{)}
	= \frac{1}{t_{n-1}} \cdot \frac{1}{1 - \frac{c_{n-1}}{c_{n}}} + 2\frac{q}{f(q)} \to 0. \qedhere
	\end{align*}
\end{proof}

\subsection{Even lower complexity}

It is natural to wonder whether the hypothesis of Theorem \ref{T:main1} is necessary.  This is, however, not the case: there exist mixing elevated rank ones with even lower complexity.

\begin{theorem}\label{T:evenlower}
\theoremtwotext
\end{theorem}
\begin{proof}
Fix $0 < \epsilon \leq 1$ and set $r_{n} = \lceil (n+1)(\log(n+1))^{1+\epsilon} \rceil - 1$ and $c_{1} = 1$ and $c_{n+1} = m_{n}$.  As $h_{n} \geq \prod_{j=1}^{n-1} r_{j} \geq \prod_{j=1}^{n-1} (j+1) = n!$ we have $\frac{r_{n}^{2}}{h_{n}} \to 0$.  By the integral comparison test, $\sum \frac{1}{r_{n}} < \infty$.  Then $\sum \frac{c_{n}}{h_{n}} < \infty$ following the same reasoning as in the proof of Theorem \ref{T:main1}.  So, by Theorem \ref{T:mixing}, the extremely elevated staircase transformation $T$ with cut sequence $\{ r_{n} \}$ and elevating sequence $\{ c_{n} \}$ is mixing on a finite measure space.

Then $c_{n} + r_{n} \leq h_{n}$ for large $n$ so $c_{n} = h_{n-1} + 2c_{n-1} + 2r_{n-1} - 2 \leq 3h_{n-1}$.  Since $1/x$ is a decreasing positive function for $x > 0$, a Riemann sum approximation gives $\sum_{q=a+1}^{b} \frac{1}{q} \geq \int_{a+1}^{b+1} \frac{1}{x}~dx = \log(b+1) - \log(a+1)$.  
Employing Proposition \ref{C:all},
\begin{align*}
\sum_{q=2}^{\infty} \frac{1}{p(q)} &= \sum_{n} \sum_{q=c_{n}+1}^{c_{n+1}} \frac{1}{p(q)}
\geq \sum_{n} \sum_{q=c_{n}+1}^{c_{n+1}} \frac{1}{q(r_{n}+1)}
= \sum_{n} \frac{1}{r_{n}+1} \sum_{q=c_{n}+1}^{c_{n+1}} \frac{1}{q} \\
&\geq \sum_{n} \frac{1}{r_{n}+1} \log\Big{(}\frac{c_{n+1}+1}{c_{n}+1}\Big{)}
\geq \sum_{n} \frac{1}{r_{n}+1} \log\Big{(}\frac{h_{n}}{3h_{n-1}}\Big{)}
\geq \sum_{n} \frac{1}{r_{n}+1} \log\Big{(}\frac{(r_{n-1}+1)h_{n-1}}{3h_{n-1}}\Big{)} \\
&\geq \sum_{n} \frac{1}{(n+1)(\log(n+1))^{1 + \epsilon}} \big{(}\log(n(\log(n))^{1+\epsilon}-1) - \log(3)\big{)} \\
&\geq \sum_{n} \frac{1}{(n+1)(\log(n+1))^{1 + \epsilon}} \big{(}\log(n) - \log(3)\big{)} \\
&= \sum_{n} \frac{1}{(n+1)(\log(n+1))^{\epsilon}}\frac{\log(n)}{\log(n+1)} - (\log(3))\sum_{n} \frac{1}{(n+1)(\log(n+1))^{1 + \epsilon}}
\end{align*}
and the left sum diverges as $\epsilon \leq 1$ while the right sum converges as $\epsilon > 0$.
\end{proof}

\subsection{Even lower complexity along sequences}

We are able to achieve even lower complexity for mixing subshifts along a sequence of lengths:

\begin{theorem}\label{T:main2}
\theoremthreetext
\end{theorem}
\begin{proof}
	Set $\alpha = \lceil (1 + \epsilon)/\epsilon \rceil$.  Since $\alpha > 1$, the function $x^{\alpha}$ is increasing so a Riemann sum approximation gives $\sum_{j=1}^{n-1} j^{\alpha} \geq \int_{0}^{n-1} x^{\alpha}~dx = (n-1)^{1+\alpha}/(1 + \alpha)$.  An easy induction argument shows $\sum_{j=1}^{n-1} j^{\alpha} \leq n^{1 + \alpha}$.  So writing $d = 1/(1 + \alpha)$, we have $d(n-1)^{1+\alpha} \leq \sum_{j=1}^{n-1} j^{\alpha} \leq n^{1+\alpha}$.
	
	Construct  $T$ inductively by setting $r_1=1$ and $c_1=1$ and, for $n > 1$,
	\[
	r_{n} = 2^{n^{\alpha}} - 1 \text{ and } c_{n} = \Big{\lceil} \frac{h_{n}}{n^{1+\epsilon}} \Big{\rceil}.
	\]
	Then $\sum \frac{c_{n}}{h_{n}} \leq \sum \frac{1}{n^{1+\epsilon}}+\frac{1}{h_{n}} < \infty$.
	Since
	\[
	\prod_{j=1}^{n-1} (r_{j} + 1) = \prod_{j=1}^{n-1} 2^{j^{\alpha}} = 2^{\sum_{j=1}^{n-1} j^{\alpha}}
	\text{ we have }
	2^{d(n-1)^{1+\alpha}} \leq \prod_{j=1}^{n-1} (r_{j}+1) \leq 2^{n^{1+\alpha}}.
	\]
	By Proposition \ref{P:hr}, we then have that for some constant $K$,
	$ 2^{d(n-1)^{1+\alpha}} \leq h_{n} \leq K \cdot 2^{n^{1+\alpha}}$.
	Then
	\[
	\frac{r_{n}^{3}}{h_{n}} \leq \frac{2^{3n^{\alpha}}}{2^{d(n-1)^{1+\alpha}}} \to 0 \quad\quad\text{since}\quad\quad 
	\frac{d(n-1)^{1+\alpha}-3n^{\alpha}}{n^{\alpha}} = d\Big{(}1-\frac{1}{n}\Big{)}^{\alpha}(n-1) - 3 \to \infty.
	\]
	To see that $T$ is an extremely elevated staircase transformation (hence is mixing on a finite measure space by Theorem \ref{T:mixing}),
	\[
	\frac{m_{n}}{c_{n+1}} \leq \frac{3h_{n}}{h_{n+1}/(n+1)^{1 + \epsilon}} \leq \frac{3h_{n}(n+1)^{1 + \epsilon}}{r_{n}h_{n}} = \frac{3(n+1)^{1+\epsilon}}{r_{n}} \to 0,
	\]
	We may apply Lemma \ref{cf4} to get $p(c_{n+1}) = p(m_{n}) + (c_{n+1} - m_{n})$.  Then Proposition \ref{C:all} gives
	\begin{align*}
	\frac{p(c_{n+1})}{h_{n+1}} \leq \frac{c_{n+1}}{h_{n+1}} + \frac{(h_{n}+2c_{n}+2r_{n}-2)(r_{n}+1)}{(r_{n}+1)h_{n}}
	\leq \frac{c_{n+1}}{h_{n+1}} + 1 + \frac{2c_{n} + 2r_{n}}{h_{n}} \to 1.
	\end{align*}
	Since $\log(c_{n}) \geq \log(h_{n}) - (1+\epsilon)\log(n) \geq \log(2^{d(n-1)^{1+\alpha}}) - 2\log(n)$,
	using that $\alpha\epsilon \geq ((1+\epsilon)/\epsilon)\epsilon = \epsilon + 1$,
	\begin{align*}
	\liminf \frac{c_{n}(\log(c_{n}))^{\epsilon}}{h_{n}} &\geq
	\liminf \frac{(d(n-1)^{1+\alpha})^{\epsilon}}{n^{1+\epsilon}}
	\geq \liminf \frac{d^{\epsilon} (n-1)^{\epsilon + \alpha\epsilon}}{n^{1+\epsilon}} \\
	& \geq \liminf \frac{d^{\epsilon}(n-1)^{1+2\epsilon}}{n^{1+\epsilon}} = \liminf d^{\epsilon}\Big{(}1- \frac{1}{n}\Big{)}^{1+\epsilon}(n-1)^{\epsilon} = \infty.
	\end{align*}
	Therefore
	\[
	\limsup \frac{p(c_{n})}{c_{n}(\log(c_{n}))^{\epsilon}} \leq \limsup \frac{p(c_{n})}{h_{n}} \limsup \frac{h_{n}}{c_{n}(\log(c_{n}))^{\epsilon}} \leq 1\cdot 0 = 0. \qedhere
	\]
\end{proof}

\subsection{A lower bound on the complexity}

Our constructions, however, do not attain complexity as low as $q \log(q)$:
\begin{theorem}\label{T:noqlogq}
\theoremfourtext
\end{theorem}
\begin{proof}
Since $T$ is extremely elevated,
$\infty > \sum_{n} \frac{c_{n+1}}{h_{n+1}} \geq \sum_{n} \frac{h_{n}}{3(r_{n}+1)h_{n}} = \frac{1}{3} \sum_{n} \frac{1}{r_{n}}
$.
By Proposition \ref{P:lowerbound},
\begin{align*}
\frac{p(m_{n})}{m_{n}\log(m_{n})} 
&\geq \frac{h_{n+1}}{3h_{n}\log(3h_{n})}
\geq \frac{r_{n}+1}{3\log(3h_{n})}. \tag{$\star$}
\end{align*}
By Proposition \ref{P:hr} there exists a constant $K$ such that $h_{n} \leq K\prod_{j=1}^{n-1}r_{j}$ so
$\log(h_{n}/K) \leq \sum_{j=1}^{n-1} \log(r_{j})$.

Consider first when $r_{n} \leq n^{2}$ for infinitely many $n$.  
Write $r_{n} + 1 = (n+1) \log(n+1) z_{n}$.  Then $z_{n} \to \infty$ since $\sum \frac{1}{r_{n}} < \infty$ and $z_{n} \leq n+1$ as we have assumed $r_{n} \leq n^{2}$,
\begin{align*}
\sum_{j=1}^{n-1} \log(r_{j})
&= \sum_{j=1}^{n-1} (\log(j+1) + \log(\log(j+1)) + \log(z_{j})) 
\leq \sum_{j=1}^{n-1} 3\log(j+1)
\leq 3n\log(n).
\end{align*}
So, as $z_{n} \to \infty$,
\[
\liminf \frac{r_{n}+1}{\log(h_{n})} \geq \liminf \frac{(n+1)\log(n+1)z_{n}}{9n\log(n)} = \liminf \frac{z_{n}}{9}
= \infty.
\]
Now consider when $r_{n} > n^{2}$ for all sufficiently large $n$.  Then as $\log(x) \leq x^{1/3}$ for large $x$ and $\log(h_{n}) \leq n\log(n+1) + \log(K)$, as $r_{n}$ is increasing,
\[
\liminf \frac{r_{n}+1}{\log(h_{n})}
\geq \liminf \frac{r_{n}+1}{n\log(r_{n}+1)} \geq \liminf \frac{r_{n}}{nr_{n}^{1/3}} = \liminf \frac{r_{n}^{2/3}}{n} \geq \liminf \frac{n^{4/3}}{n} = \infty.
\]
In both cases, we have $\liminf \frac{r_{n}+1}{\log(h_{n})} \to \infty$.  By equation $(\star)$, this completes the proof.
\end{proof}

\subsection{Linear complexity is unattainable even along a sequence}

Though the complexity along a sequence can be lower than $q \log(q)$, it cannot be linear:

\begin{theorem}\label{T:superlinear}
\theoremfivetext
\end{theorem}
\begin{proof}
Let $\epsilon > 0$.  Then there exists $N$ such that for $n \geq N$, we have $\frac{c_{n}+r_{n}}{h_{n}} < \epsilon$ (since $T$ is on a finite measure space) and $r_{n} \geq 1/\epsilon$ (since $r_{n} \to \infty$ is necessary for $T$ to be mixing).

For $q \geq m_{N-1}$, choose $n \geq N$ such that $m_{n-1} \leq q < m_{n}$.

If $m_{n-1} \leq q < 2(c_{n} + r_{n})$ then, using Proposition \ref{P:lowerbound},
\[
\frac{p(q)}{q} \geq \frac{p(m_{n-1})}{2(c_{n}+r_{n})} \geq \frac{h_{n}}{2(c_{n}+r_{n})} > \frac{1}{2\epsilon}.
\]

For $c_{n} + r_{n} \leq q < h_{n} + 2c_{n}$, by Lemma \ref{cf2}, $p(q) - p(c_{n}+r_{n}) \geq (q - c_{n} - r_{n})r_{n}$.  Then for $2(c_{n} + r_{n}) \leq q < h_{n} + 2c_{n} + 1$,
\[
\frac{p(q)}{q} \geq \frac{(q - c_{n} - r_{n})r_{n}}{q}
\geq \Big{(}1 - \frac{c_{n}+r_{n}}{q}\Big{)}r_{n}
\geq \frac{1}{2}r_{n} 
> \frac{1}{2\epsilon}.
\]

For $h_{n} + 2c_{n} + 1 \leq q < m_{n}$, we have $p(q) \geq p(h_{n} + 2c_{n}) \geq (h_{n} + c_{n} - r_{n})r_{n}$.  Provided $\epsilon < 1/4$, we have $(1 - \epsilon)/(1+2\epsilon) \geq 1/2$ so for $h_{n} + 2c_{n} \leq q < m_{n}$,
\[
\frac{p(q)}{q} \geq \frac{(h_{n}+c_{n}-r_{n})r_{n}}{m_{n}}
= \frac{1 + \frac{c_{n}-r_{n}}{h_{n}}}{1 + 2\frac{c_{n}+r_{n}-1}{h_{n}}} \cdot r_{n}
> \frac{1-\epsilon}{1 + 2\epsilon} \cdot \frac{1}{\epsilon}
\geq \frac{1}{2\epsilon}.
\]

Taking $\epsilon \to 0$ then gives $\frac{p(q)}{q} \to \infty$ as for all sufficiently large $q$ we have $\frac{p(q)}{q} > \frac{1}{2\epsilon}$.
\end{proof}

\appendix

\section{Mixing for extremely elevated staircase transformations}

For our proof of mixing, we do not need the full strength of extremely elevated staircase transformations and so will define a  more general class:

\begin{definition}
A rank-one transformation is an \textbf{elevated staircase transformation} when it has nondecreasing cut sequence $\{ r_{n} \}$ tending to infinity with $\frac{r_{n}^{2}}{h_{n}} \to 0$, and spacer sequence given by $s_{n,i} = c_{n} + i$ for $0 \leq i < r_{n}$ and $s_{n,r_n} = 0$ for some sequence $\{ c_{n} \}$ such that $c_{n+1} \geq c_{n} + r_{n}$ and $\sum \frac{c_{n} + r_{n}}{h_{n}} < \infty$.
\end{definition}

This is the same class as the more natural $s_{n,i} = e_{n} + i$ for a sequence $\{ e_{n} \}$ required to satisfy no condition beyond $e_{n} \geq 0$ (and $\sum \frac{1}{h_{n}}\sum_{j \leq n} e_{j} < \infty$ to ensure finite measure).  In particular, traditional staircases, corresponding to $e_{n} = 0$, are in the class of elevated staircase transformations.

\begin{proposition}\label{P:isomo}
Let $\{ e_{n} \}$ be a sequence of nonnegative integers.  Let $\tilde{T}$ be the rank-one transformation with cut sequence $\{ r_{n} \}$ and spacer sequence $\{ \tilde{s}_{n,i} \}$ given by $\tilde{s}_{n,i} = e_{n} + i$ for $0 \leq i \leq r_{n}$.  Let $T$ be the rank-one transformation with cut sequence $\{ r_{n} \}$ and elevating sequence $\{ c_{n} \}$ given by $c_{1} = e_{1}$ and $c_{n+1} = e_{n+1} + \sum_{j=1}^{n} (e_{j} + r_{j}) = e_{n+1} + c_{n} + r_{n}$ and spacer sequence given by $s_{n,i} = c_{n}+i$ for $0 \leq i < r_{n}$ and $s_{n,r_{n}} = 0$.  Then $T$ and $\tilde{T}$ generate the same subshift (and are measure-theoretically isomorphic).
\end{proposition}
\begin{proof}
If $\tilde{B}_{n}$ are the words representing the $\tilde{s}_{n,i}$ construction and $B_{n}$ those of $T$ then $\tilde{B}_{1} = B_{1} = 0$ and $\tilde{B}_{n+1} = \prod_{i=1}^{r_{n}} \tilde{B}_{n}1^{e_{n}+i}$ and $B_{n} = (\prod_{i=0}^{r_{n}-1}B_{n}1^{c_{n}+i})B_{n}$ and we claim that $\tilde{B}_{n+1} = B_{n+1}1^{\sum_{j=1}^{n}(e_{j}+r_{j})}$ for all $n \geq 1$.  The base case is
\[
\tilde{B}_{2} = \prod_{i=0}^{r_{1}}\tilde{B}_{1}1^{e_{1}+i}
= \Big{(}\prod_{i=0}^{r_{1}-1}\tilde{B}_{1}1^{e_{1}+i}\Big{)}\tilde{B}_{1}1^{e_{1}+r_{1}}
= \Big{(}\prod_{i=0}^{r_{1}-1}B_{1}1^{c_{1}+i}\Big{)}B_{1}1^{e_{1}+r_{1}}
\]
as claimed since $c_{1} = e_{1}$.  Assume the claim holds for $n$ and then
\begin{align*}
\tilde{B}_{n+2} &= \prod_{i=0}^{r_{n+1}}\tilde{B}_{n+1}1^{e_{n+1}+i}
= \Big{(}\prod_{i=0}^{r_{n+1}-1}\tilde{B}_{n+1}1^{e_{n+1}+i}\Big{)}\tilde{B}_{n+1}1^{e_{n+1}+r_{n+1}} \\
&= \Big{(}\prod_{i=0}^{r_{n+1}-1}B_{n+1}1^{\sum_{j=1}^{n}(e_{j}+r_{j})+e_{n+1}+i}\Big{)} B_{n+1}1^{\sum_{j=1}^{n}(e_{j}+r_{j})+e_{n+1}+r_{n+1}} \\
&= \Big{(}\prod_{i=0}^{r_{n+1}-1}B_{n+1}1^{c_{n+1}+i}\Big{)} B_{n+1}1^{\sum_{j=1}^{n+1}(e_{j}+r_{j})}
\end{align*}
so the claim holds for all $n$.  As this means every subword of $\tilde{B}_{n}$ is a subword of $B_{n}$ or $B_{n+1}$ and conversely (with $\tilde{B}_{n-1}$ rather than $\tilde{B}_{n+1}$), the languages of the transformations are the same.
\end{proof}

The proof of mixing is very similar to that of \cite{CreutzSilva2004} for traditional staircases; our proof is self-contained.

Theorem \ref{T:mixing} is a special case of:
\begin{theorem}\label{T:realmixing}
Every elevated staircase transformation is mixing (on a finite measure space).
\end{theorem}

\begin{remark} The requirement that $\frac{r_{n}^{2}}{h_{n}} \to 0$ is not necessary but one would need to bring the more complicated and technical techniques of \cite{CreutzSilva2010} in to prove it. \end{remark}

The remainder of the appendix is devoted the proof of Theorem \ref{T:realmixing}.

\begin{proposition}\label{P:finmeasest}
Every elevated staircase transformation is on a finite measure space.
\end{proposition}
\begin{proof}
	Writing $S_{n}$ for the union of the spacers added above the $n^{th}$ column,
	\[
	\mu(S_{n}) = (c_{n}r_n+\frac{1}{2}r_{n}(r_{n}-1))\mu(I_{n+1}) = \left(c_n\frac{r_n}{r_{n}+1}+\frac{1}{2}\frac{r_{n}(r_{n}-1)}{r_{n}+1}\right)\mu(I_{n}) \leq \frac{c_{n} + r_{n}}{h_{n}}~\mu(C_{n}),
	\]
	and therefore
	$
	\mu(C_{n+1}) = \mu(C_{n}) + \mu(S_{n}) \leq \big{(} 1 + \frac{c_{n} + r_{n}}{h_{n}} \big{)} \mu(C_{n}).
	$.
	Then
	$
	\mu(C_{n+1}) \leq \prod_{j=1}^{n} \big{(}1 + \frac{c_{j} + r_{j}}{h_{j}}\big{)}\mu(C_{1}),
	$
	meaning that
	$
	\log(\mu(C_{n+1})) \leq \log(\mu(C_{1})) + \sum_{j=1}^{n} \log(1 + \frac{c_{j} + r_{j}}{h_{j}}).
	$
	As $\frac{c_{n} + r_{n}}{h_{n}} \to 0$, since $\log(1+x) \approx x$ for $x \approx 0$,
	$
	\lim_{n} \log(\mu(C_{n+1})) \lesssim \log(\mu(C_{1})) + \sum_{j=1}^{\infty} \frac{c_{j} + r_{j}}{h_{j}} < \infty
	$
	gives that $T$ is on a finite measure space.
	\end{proof}
	
From here on, assume that all transformations $T$ are on probability spaces.
	
\begin{lemma}\label{L:subtofull}
Let $T$ be any rank-one transformation and $B$ be a union of levels in some column $C_{N}$.  Then for any $n \geq N$, $0 \leq a < h_{n}$ and $0 \leq i \leq r_{n}$,
\[
\mu(I_{n,a}^{[i]} \cap B) - \mu(I_{n,a}^{[i]})\mu(B) = \frac{1}{r_{n}+1}\big{(}\mu(I_{n,a} \cap B) - \mu(I_{n,a})\mu(B)\big{)}.
\]
\end{lemma}
\begin{proof}
Since $B$ is a union of levels in $C_{N}$, it is also a union of levels in $C_{n}$.  Therefore $I_{n,a} \subseteq B$ or $I_{n,a} \cap B = \varnothing$.  When $I_{n,a} \subseteq B$, we have $\mu(I_{n,a}^{[i]} \cap B) = \mu(I_{n,a}^{[i]}) = \frac{1}{r_{n}+1}\mu(I_{n,a}) = \frac{1}{r_{n}+1}\mu(I_{n,a} \cap B)$ and when $I_{n,a} \cap B = \varnothing$, we have $\mu(I_{n,a}^{[i]} \cap B) = 0 = \mu(I_{n,a} \cap B)$.
\end{proof}

\begin{lemma} \label{inequalforMix} Let $T$ be an elevated staircase transformation with height sequence $\{h_n\}$. Let $I_{n,a}$ be the $a^{th}$ level in the $n^{th}$ column $C_{n}$ for $T$. Let $B$ be a union of levels in a column $C_{N}$ with $N \leq n$. Then for $k$ such that $ ki+\frac{1}{2}k(k-1) \leq a <h_n$,
	\begin{align}
	|\mu(T^{k(h_n+c_n)}(I_{n,a}) \cap B) - \mu(I_{n,a})\mu(B)| \leq \int \limits_{I_{n,a}} \Big{|}\frac{1}{r_n+1} \sum \limits_{i=0}^{r_n} \chi_{B} \circ T^{-ki-\frac{1}{2}k(k-1)}-\mu(B)\Big{|} d\mu + \frac{2k+2}{r_n+1}\mu(I_{n}). \nonumber
	\end{align}
\end{lemma}	

\begin{proof} Write $I_{n,a}$ as a disjoint union of all the sublevels of $I_{n,a}$ so that
	\begin{align}
	|\mu(T^{k(h_n+c_n)}(I_{n,a}) \cap B) &- \mu(I_{n,a})\mu(B)| =|\sum \limits_{i=0}^{r_{n}} \mu(T^{k(h_n+c_n)}(I^{[i]}_{n,a}) \cap B) - \mu(I^{[i]}_{n,a})\mu(B)|. \nonumber
	\end{align}

Now for $i<r_n$, $T^{h_n}(I_{n,a}^{[i]})=T^{-i-c_n}(I_{n,a}^{[i+1]})$ and so $T^{h_n+c_n}(I_{n,a}^{[i]})=T^{-i}(I_{n,a}^{[i+1]})$. Applying this $k$ times, for $i < r_n-k$, we get $T^{k(h_n+c_n)}(I_{n,a}^{[i]}) = T^{-i-(i+1)-...-(i+k-1)}(I_{n,a}^{[i+k]})=T^{-ki-\frac{1}{2}k(k-1)}(I_{n,a}^{[i+k]})$. So for $ki+\frac{1}{2}k(k-1) \leq a <h_n$,
	\begin{align}
	|\mu(T^{k(h_n+c_n)}&(I_{n,a}) \cap B) - \mu(I_{n,a})\mu(B)| \nonumber
	= |\sum_{i=0}^{r_n}\mu(T^{k(h_n + c_n)}(I_{n,a}^{[i]}) \cap B) - \mu(I_{n,a}^{[i]})\mu(B)| \\
	&\leq |\sum \limits_{i=0}^{r_{n}-(k+1)} \mu(T^{-ki-\frac{1}{2}k(k-1)}(I^{[i+k]}_{n,a}) \cap B) - \mu(I^{[i+k]}_{n,a})\mu(B)| + \frac{k+1}{r_{n}+1}\mu(I_{n,a}) \nonumber \\
	&= |\sum \limits_{i=0}^{r_{n}-(k+1)} \mu(I^{[i+k]}_{n,a-ki-\frac{1}{2}k(k-1)} \cap B) - \mu(I^{[i+k]}_{n,a-ki-\frac{1}{2}k(k-1)})\mu(B)| + \frac{k+1}{r_n+1}\mu(I_{n,a}). \nonumber 
	\end{align}
	By Lemma \ref{L:subtofull} then
	\begin{align}
	|\mu(T^{k(h_n+c_n)}&(I_{n,a}) \cap B) - \mu(I_{n,a})\mu(B)| \nonumber \\
	&\leq  |\frac{1}{r_n+1}\sum \limits_{i=0}^{r_{n}-(k+1)}\mu(I_{n,a-ki-\frac{1}{2}k(k-1)} \cap B) - \mu(I_{n,a-ki-\frac{1}{2}k(k-1)})\mu(B) | + \frac{k+1}{r_n+1}\mu(I_{n,a}) \nonumber  \\
	&= |\frac{1}{r_n+1} \sum \limits_{i=0}^{r_n-(k+1)}\mu(T^{-ki-\frac{1}{2}k(k-1)}(I_{n,a}) \cap B) - \mu(I_{n,a})\mu(B)| + \frac{k+1}{r_n+1}\mu(I_{n,a}) \nonumber  \\
	&\leq |\frac{1}{r_n+1} \sum \limits_{i=0}^{r_n}\mu(T^{-ki-\frac{1}{2}k(k-1)}(I_{n,a}) \cap B) - \mu(I_{n,a})\mu(B)| + 2\frac{k+1}{r_n+1}\mu(I_{n,a}) \nonumber \\
	&\leq \int \limits_{I_{n,a}} \Big{|}\frac{1}{r_n+1} \sum \limits_{i=0}^{r_n} \chi_{B} \circ T^{-ki-\frac{1}{2}k(k-1)}-\mu(B)\Big{|} d\mu + \frac{2k+2}{r_n+1}\mu(I_{n,a}). \qedhere \nonumber
	\end{align}
\end{proof}

\begin{definition}
A sequence $\{ t_{n} \}$ is \textbf{mixing} for $T$ when for all measurable sets $A$ and $B$,
\[
\lim_{n \to \infty} \mu(T^{n}A \cap B) = \mu(A)\mu(B).
\]
\end{definition}

\begin{definition}[\cite{CreutzSilva2004}]
	A sequence $\{t_n\}$ is \textbf{rank-one uniform mixing} for $T$ when for every union of levels $B$, 
	\[ \lim_{n \to \infty}\sum \limits_{a=0}^{h_{n}-1} |\mu(T^{t_n}(I_{n,a}) \cap B) - \mu(I_{n,a})\mu(B)| = 0 .\]
\end{definition}

\begin{proposition}[\cite{CreutzSilva2004}] \label{UniformToMixing} If $\{t_n\}$ is rank-one uniform mixing for $T$, then $\{t_n\}$ is mixing for $T$.\end{proposition}
\begin{proof} 
Every measurable set can be arbitrarily well approximated by a union of levels.
\end{proof}

\begin{theorem}\label{kheightUniMixElevated} Let $T$ be an elevated staircase transformation with height sequence $\{h_n\}$ and $k \: \in \mathbb{N}$ such that $T^k$ is ergodic. Then the sequence $\{k(h_n+c_n)\}$ is rank-one uniform mixing for $T$.
 \end{theorem}

\begin{proof} By Lemma \ref{inequalforMix}, for $a$ such that $ ki+\frac{1}{2}k(k-1)\leq a <h_n$, since $ki + \frac{1}{2}k(k-1) \leq kr_n + k^{2}$,
	\begin{align}
	\sum \limits_{a=0}^{h_{n}-1} &|\mu(T^{k(h_n+c_n)}(I_{n,a}) \cap B) - \mu(I_{n,a})\mu(B)| \nonumber \\
	&\leq (kr_{n}+k^2)\mu(I_n) + \sum \limits_{a=kr_{n}+r_n^2}^{h_n-1} \left(\int_{I_{n,a}} \Big{|}\frac{1}{r_n+1} \sum \limits_{i=0}^{r_n} \chi_{B} \circ T^{-ki-\frac{1}{2}k(k-1)}-\mu(B)\Big{|} d\mu + \frac{2k+2}{r_n+1}\mu(I_{n,a}) \right) \nonumber \\
	&\leq (kr_{n}+k^2)\mu(I_{n}) + \int \Big{|}\frac{1}{r_n+1} \sum \limits_{i=0}^{r_n} \chi_{B} \circ T^{-ki-\frac{1}{2}k(k-1)}-\mu(B)\Big{|} d\mu +h_{n}\left(\frac{2k+2}{r_n+1}\right)\mu(I_{n}), \nonumber
	\end{align}
	using that the levels are disjoint.  Clearly $(kr_n + k^2)\mu(I_{n}) \leq \frac{kr_{n}}{h_{n}}+\frac{k^2}{h_{n}} \to 0$ and $h_{n}\frac{2k+2}{r_{n}+1}\mu(I_{n}) \leq \frac{2k+2}{r_{n}+1} \to 0$.  
	That $T$ is measure-preserving and the mean ergodic theorem applied to $T^{k}$ give
	\begin{align*}
	\int \Big{|} \frac{1}{r_n+1} \sum \limits_{i=0}^{r_n} \chi_{B} \circ T^{-ki-\frac{1}{2}k(k-1)}-\mu(B\Big{|}| d\mu &\leq \int \Big{|}\frac{1}{r_n+1} \sum \limits_{i=0}^{r_n} \chi_{B} \circ T^{-ki}-\mu(B)\Big{|} d\mu \\
	&\leq \Big{(} \int\Big{|}\frac{1}{r_n+1} \sum \limits_{i=0}^{r_n} \chi_{B} \circ T^{-ki}-\mu(B)\Big{|}^2 d\mu\Big{)}^{1/2}
	\to 0. \qedhere \end{align*}
	 \end{proof}

\begin{corollary}\label{C:Tkerg}
	If $T$ is an elevated staircase transformation then $T^k$ is ergodic for each fixed $k$.
\end{corollary}

\begin{proof}
	Using Theorem \ref{kheightUniMixElevated} with $k=1$, since $T$ is ergodic we have that $\{h_n+c_n\}$ is uniform mixing, hence mixing by Proposition \ref{UniformToMixing}.  The existence of a mixing sequence for $T$ implies $T$ is weakly mixing hence each power of $T$ is ergodic.
\end{proof}

\begin{lemma} \label{between} 
Let $T$ be a rank-one transformation and $\{ c_{n} \}$ a sequence such that $\frac{c_{n}}{h_{n}} \to 0$.  If $q\in\mathbb{N}$ and $\{ q(h_{n}+c_{n}) \}$ and $\{ (q+1)(h_{n} + c_{n}) \}$ are rank-one uniform mixing and $\{ t_{n} \}$ is a sequence such that $q(h_{n} + c_{n}) \leq t_{n} < (q+1)(h_{n}+c_{n})$ for all $n$ then $\{ t_{n} \}$ is rank-one uniform mixing.
\end{lemma}

\begin{proof} 
For $0 \leq a < q(h_n+c_n)-t_n+h_n$, we have $0 \leq t_n-q(h_n+c_n) \leq t_n+a-q(h_n+c_n)<h_n$, so
	\[T^{t_n}(I_{n,a})=T^{t_n+a}(I_{n,0})=T^{q(h_n+c_n)}(I_{n,t_n+a-q(h_n+c_n)}).\]
For $(q+1)(h_n+c_n)-t_n \leq a < h_n$, we have $0 \leq t_n+a-(q+1)(h_n+c_n)<a<h_n$, so 
\[T^{t_n}(I_{n,a})=T^{t_n+a}(I_{n,0})=T^{(q+1)(h_n+c_n)}(I_{n,t_n+a-(q+1)(h_n+c_n)}).\]
For a union of levels $B$ in $C_N$ and $n \geq N$,
\begin{align}
	\sum \limits_{a=0}^{h_n-1} &|\mu(T^{t_n}(I_{n,a} \cap B) - \mu(I_{n,a})\mu(B)|  \nonumber \\
&\leq \sum \limits_{a=0}^{q(h_n+c_n)-t_n+h_n-1} |\mu(T^{q(h_n+c_n)}I_{n,t_n+a-q(h_n+c_n)} \cap B) - \mu(I_{n})\mu(B)| + c_{n}\mu(I_{n}) \nonumber \\
&\quad\quad\quad\quad  +\sum \limits_{a=(q+1)(h_n+c_n)-t_n}^{h_n-1} |\mu(T^{(q+1)(h_n+c_n)}I_{n,t_n+a-(q+1)(h_n+c_n)} \cap B) - \mu(I_{n})\mu(B)| \nonumber \\
&\leq \sum \limits_{b=0}^{h_n-1} |\mu(T^{q(h_n+c_n)}I_{n,b} \cap B) - \mu(I_{n})\mu(B)| + c_{n}\mu(I_{n}) \nonumber \\
&\quad\quad\quad\quad + \sum \limits_{b=0}^{h_n-1} |\mu(T^{(q+1)(h_n+c_n)}I_{n,b} \cap B) - \mu(I_{n})\mu(B) | \to 0 \nonumber 
\end{align}	
since $\{q(h_n+c_n)\}$, $\{(q+1)(h_n+c_n)\}$ are rank-one uniform mixing and $c_{n}\mu(I_{n}) \leq \frac{c_{n}}{h_{n}} \to 0$.
\end{proof} 

\begin{proposition}\label{P:CCC}
Let $T$ be a rank-one transformation and $\{ c_{n} \}$ a sequence such that $\frac{c_{n}}{h_{n}} \to 0$.  If $k \in \mathbb{N}$ and $\{ q(h_{n}+c_{n}) \}$ is rank-one uniform mixing for each $q \leq k+1$ and $\{ t_{n} \}$ is a sequence such that $h_{n} + c_{n} \leq t_{n} < (k+1)(h_{n}+c_{n})$ for all $n$ then $\{ t_{n} \}$ is mixing.
\end{proposition}

\begin{proof}
	Since $t_{n} < (k+1)(h_n+c_n)$, there is some $q_n \leq q$ such that $q_n(h_n+c_n) \leq t_n < (q_n+1)(h_n+c_n)$. Let $\{t_{n_j}\}$ be any subsequence of $\{t_n\}$. Since $q_n \leq k$ for all $n$ and $q$ is fixed, there exists a further subsequence $\{t_{n_{j_k}}\}$ on which $q_{n_{j_k}}$ is constant. By Lemma \ref{between} and Proposition \ref{UniformToMixing}, $\{t_{n_{j_k}}\}$ is mixing. As every subsequence of $\{t_n\}$ has a mixing subsequence, $\{t_n\}$ is mixing.
\end{proof}

\begin{lemma}\label{L:tnfixedl2} Let $T$ be a measure-preserving transformation.  If for each  fixed $\ell \in \mathbb{N}$, $\{\ell t_n\}$ is mixing, then for any $\epsilon > 0$ there exists $L$ and $N$ such that for all $n \geq N$,
	$\int |\frac{1}{L} \sum \limits_{\ell=1}^{L} \chi_{B} \circ T^{-\ell t_n}-\mu(B)| d\mu < \epsilon.$
\end{lemma}
\begin{proof}
Take $L > 2/\epsilon^{2}$ and $N$ so that $|\mu(T^{\ell t_{n}}(B) \cap B) - \mu(B)\mu(B)| < \epsilon^{2}/2$ for $\ell < L$ and $n > N$.  Then
\begin{align*}
\int \Big{|}\frac{1}{L} \sum \limits_{m=1}^{L} \chi_{B} \circ T^{-m t_n} - \mu(B)\Big{|}^{2} d\mu
&= \frac{1}{L^{2}}\sum \limits_{r,m=1}^{L} \mu(T^{(m-r)t_{n}}(B) \cap B) - \mu(B)\mu(B) \\
&\leq \frac{1}{L} + \frac{1}{L} \sum \limits_{\ell=1}^{L-1} \frac{L-\ell}{L} \mu(T^{\ell t_{n}}(B) \cap B) - \mu(B)\mu(B) < 2\epsilon^{2}/2 = \epsilon^{2}
\end{align*}
so, by Cauchy-Schwarz, $\int |\frac{1}{L} \sum \limits_{\ell=1}^{L} \chi_{B} \circ T^{-\ell t_n}-\mu(B)| d\mu \leq \sqrt{\epsilon^{2}} = \epsilon$.
\end{proof}

\begin{lemma}[Block Lemma \cite{adams1998smorodinsky}] \label{BlockLemma} For $T$ measure-preserving and $R,L,p \in \mathbb{N}$ with $pL \leq R$,\\
	$
	\int |\frac{1}{R} \sum \limits_{r=0}^{R-1} \chi \circ T^{-r}| d\mu \leq \int |\frac{1}{L} \sum \limits_{\ell=0}^{L-1} \chi \circ T^{-p\ell}| d\mu +\frac{pL}{R} \int |\chi| d\mu.
	$
\end{lemma}
\begin{proof}
$0 \leq R - pL\lfloor \frac{R}{pL} \rfloor \leq \frac{pL}{r}$ so $\int |\frac{1}{R} \sum \limits_{r=0}^{R-1} \chi \circ T^{-r}| d\mu \leq \frac{pL}{R} + \int |\frac{1}{R} \sum \limits_{r=0}^{\lfloor \frac{R}{pL}\rfloor-1} \chi \circ T^{-r}| d\mu$ and
\begin{align*}
\int \Big{|}\frac{1}{R}\sum \limits_{r=0}^{pL\lfloor \frac{R}{pL} \rfloor -1} \chi \circ T^{-r}\Big{|} d\mu 
&= \frac{pL\lfloor \frac{R}{pL} \rfloor}{R} \int \Big{|}\frac{1}{\lfloor \frac{R}{pL} \rfloor}\sum \limits_{m=0}^{\lfloor \frac{R}{pL} \rfloor-1} \frac{1}{p}\sum \limits_{b=0}^{p-1} \frac{1}{L}\sum \limits_{\ell=0}^{L-1} \int \chi \circ T^{-p\ell} \circ T^{-b} \circ T^{-mpL}\Big{|} d\mu \\
&\leq  \frac{1}{\lfloor \frac{R}{pL} \rfloor}\sum \limits_{m=0}^{\lfloor \frac{R}{pL} \rfloor-1} \frac{1}{p}\sum \limits_{b=0}^{p-1} \int\Big{|} \frac{1}{L}\sum \limits_{\ell=0}^{L-1} \chi \circ T^{-p\ell}\circ T^{-b} \circ T^{-mpL}  \Big{|} d\mu \\
&= \frac{1}{\lfloor \frac{R}{pL} \rfloor}\sum \limits_{m=0}^{\lfloor \frac{R}{pL} \rfloor-1} \frac{1}{p}\sum \limits_{b=0}^{p-1} \int\Big{|} \frac{1}{L}\sum \limits_{\ell=0}^{L-1} \chi \circ T^{-p\ell}  \Big{|} d\mu
= \int \Big{|}\frac{1}{L}\sum \limits_{\ell=0}^{L-1} \chi \circ T^{-p\ell}  \Big{|} d\mu. \qedhere
\end{align*}
\end{proof}

\begin{proposition}\label{pe} 
Let $T$ be a rank-one transformation and $\{ c_{n} \}$ a sequence such that $\frac{c_{n}}{h_{n}} \to 0$.  If $\{ q(h_{n}+c_{n}) \}$ is rank-one uniform mixing for each fixed $q$ and $k_n \to \infty$ is such that $\frac{k_n}{n} \leq 1$ then
	\[
	\int \Big{|}\frac{1}{n} \sum \limits_{j=0}^{n-1} \chi \circ T^{-jk_n}\Big{|} d\mu \rightarrow 0.
	\]
	This condition is called power ergodic in \cite{CreutzSilva2004} and \cite{CreutzSilva2010}.
\end{proposition}

\begin{proof} For each $n$ there exists a unique $m$ such that $h_m+c_m \leq k_n < h_{m+1}+c_{m+1}$. Let $p_n$ be the smallest integer such that $p_nk_n \geq h_{m+1}+c_{m+1}$.
	Suppose $p_nk_n > 2(h_{m+1}+c_{m+1})$. Then $(\frac{p_n}{2})k_n > h_{m+1}+c_{m+1}$. If $p_n$ is even, $p_n > \frac{p_n}{2}$, which contradicts that $p_n$ is the smallest integer such that $p_nk_n\geq h_{m+1}+c_{m+1}$. If $p_n$ is odd, $p_n \geq \frac{p_n+1}{2}$, which contradicts that $p_n$ is smallest such that $p_nk_n\geq h_{m+1}+c_{m+1}$. In the case when $p_n=1$, then $k_n \geq 2(h_{m+1}+c_{m+1})$ with $k_n =h_{m+1}+c_{m+1}$, contradicting that $k_{n} < h_{m+1} + c_{m+1}$. So $p_nk_n < 2(h_{m+1}+c_{m+1})$.
	Set $t_n=p_nk_n$. Then $h_{m+1}+c_{m+1} \leq t_n < 2(h_{m+1}+c_{m+1})$.  For each fixed $\ell$ then $(h_{m} + c_{m}) \leq \ell t_{n} < 2\ell (h_{m} + c_{m})$ so $\{ \ell t_{n} \}$ is mixing by Proposition \ref{P:CCC}.
	
	Fix $\epsilon > 0$.  By Lemma \ref{L:tnfixedl2}, there exists $L$ and $N$ such that for $n > N$,
	$
	\int |\frac{1}{L} \sum \limits_{\ell=1}^{L} \chi \circ T^{-\ell t_n}| d\mu < \epsilon.
	$
	By Lemma \ref{BlockLemma},
	\begin{align}
	\int \Big{|}\frac{1}{n} \sum \limits_{j=0}^{n-1} \chi \circ T^{-jk_n}\Big{|} d\mu &\leq \int \Big{|}\frac{1}{L} \sum \limits_{\ell=0}^{L-1} \chi \circ T^{-\ell p_{n}k_n}\Big{|} d\mu + \frac{p_n L}{n} \nonumber 
	= \int \Big{|}\frac{1}{L} \sum \limits_{\ell=0}^{L-1} \chi \circ T^{-\ell t_n}\Big{|} d\mu + \frac{p_n L}{n} \nonumber 
	< \epsilon +\frac{p_n L}{n}, \nonumber 
	\end{align}	
	
	Since $\frac{k_{n}}{n} \leq 1$ gives $\frac{r_{m}}{n} = \frac{r_{m}k_{n}}{nk_{n}} \leq \frac{r_{m}}{k_{n}} \leq \frac{r_{m}}{h_{m}} \to 0$,
	\[
	\frac{p_{n}}{n} = \frac{p_{n}k_{n}}{nk_{n}} \leq \frac{2(h_{m+1}+c_{m+1})}{n(h_{m}+c_{m})} \leq \frac{4}{n} \frac{(r_{m}+1)(h_{m} + c_{m} + r_{m})}{(h_{m}+c_{m})}
	= \frac{4r_{m}}{n} \Big{(}1 + \frac{r_{m}}{h_{m}+c_{m}}\Big{)} \to 0
	\]
	so
	$
	\limsup_{n} \int |\frac{1}{n} \sum \limits_{j=0}^{n-1} \chi \circ T^{-jk_n}| d\mu \leq \epsilon.
	$
	As this holds for all $\epsilon > 0$, $\int |\frac{1}{n} \sum \limits_{j=0}^{n-1} \chi \circ T^{-jk_n}| d\mu \to 0$.
\end{proof}

\begin{theorem}\label{mixingES} Let $T$ be an elevated staircase transformation with height sequence $\{h_n\}$ such that $\frac{r_n^2}{h_n}\rightarrow 0$. Let $\{t_n\}$ be a sequence such that $(h_n+c_n) \leq t_n < (h_{n+1}+c_{n+1})$. Then $\{t_n\}$ is mixing. \end{theorem}

\begin{proof} 

By Corollary \ref{C:Tkerg}, $T^{k}$ is ergodic for each fixed $k$.  Then by Theorem \ref{kheightUniMixElevated}, the sequence $\{ k(h_{n} + c_{n}) \}$ is rank-one uniform mixing for each fixed $k$.  By Proposition \ref{P:CCC}, if there exists a  constant $k$ such that $(h_n+c_n) \leq t_n < k(h_n+c_n)$, then $\{t_n\}$ is mixing, so writing $t_{n} = k_{n}(h_{n}+c_{n}) + z_{n}$ for $0 \leq z_{n} < h_{n} + c_{n}$ we may assume $k_n \to \infty$.
	
For $0 \leq a < h_n-z_n$, we have $T^{t_n}(I_{n,a})=T^{k_n(h_n+c_n)}(I_{n,a+z_n})$ and for $h_n+c_n-z_n \leq a < h_n$, 
\[T^{t_n}(I_{n,a})=T^{t_n+a}(I_{n,0})=T^{k_n(h_n+c_n)+z_n+a}(I_{n,0})=T^{(k_n+1)(h_n+c_n)}(I_{n,a+z_n-h_n-c_n}).\]

For a union of levels $B$ in $C_N$ and $n \geq N$,
	\begin{align}
	&\sum \limits_{a=0}^{h_n-1} |\mu(T^{t_n}(I_{n,a}) \cap B) - \mu(I_{n,a})\mu(B)|  \nonumber \\
	&\leq \sum \limits_{a=0}^{h_n-z_{n}-1} |\mu(T^{t_n}(I_{n,a}) \cap B) - \mu(I_{n,a})\mu(B)| + c_{n}\mu(I_{n}) + \sum \limits_{a=h_{n}+c_{n}+z_{n}}^{h_n-1} |\mu(T^{t_n}(I_{n,a}) \cap B) - \mu(I_{n,a})\mu(B)| \nonumber \\
	&\leq \sum \limits_{b=0}^{h_n-1} |\mu(T^{k_n(h_n+c_n)}(I_{n,b}) \cap B) - \mu(I_{n,b})\mu(B)| + c_{n}\mu(I_{n}) \tag{$\star$} \nonumber  \\
	&\quad\quad\quad\quad +\sum \limits_{b=0}^{h_n-1} |\mu(T^{(k_n+1)(h_n+c_n)}(I_{n,b}) \cap B) - \mu(I_{n,b})\mu(B)|.  \tag{$\star\star$} \nonumber 
	\end{align}
	
	We show that sum $(\star)$ tends to zero:
	\begin{align}
	\sum \limits_{b=0}^{h_n-1} |\mu(&T^{k_n(h_n+c_n)}(I_{n,b}) \cap B) - \mu(I_{n,b})\mu(B)| \nonumber 
	\leq \sum \limits_{b=0}^{h_n-1} |\sum \limits_{i=0}^{r_{n}-k_n}\mu(T^{k_n(h_n+c_n)}(I^{[i]}_{n,b}) \cap B) - \mu(I^{[i]}_{n,b})\mu(B)|  \nonumber \tag{$\dagger$} \\ &\quad\quad\quad\quad + \frac{2}{r_{n}} + \sum \limits_{b=0}^{h_n-1} |\sum \limits_{i=r_{n}-k_n+2}^{r_{n}}\mu(T^{k_n(h_n+c_n)}(I^{[i]}_{n,b}) \cap B) - \mu(I^{[i]}_{n,b})\mu(B)|. \nonumber \tag{$\ddagger$}
	\end{align}
For the sum $(\dagger)$,
	\begin{align}
	\sum \limits_{b=0}^{h_n-1} |\sum \limits_{i=0}^{r_{n}-k_n}&\mu(T^{k_n(h_n+c_n)}(I^{[i]}_{n,b}) \cap B) - \mu(I^{[i]}_{n,b})\mu(B)|  \nonumber
	\leq \left(r_nk_n+ \frac{1}{2}k_n(k_n-1)\right)\mu(I_n) \nonumber \\
	&\quad \quad \quad \quad + \sum \limits_{b=r_nk_n + \frac{1}{2}k_n(k_n-1)}^{h_n-1} |\sum \limits_{i=0}^{r_{n}-k_n}\mu(T^{k_n(h_n+c_n)}(I^{[i]}_{n,b}) \cap B) - \mu(I^{[i]}_{n,b})\mu(B)|, \nonumber 
	\end{align}	
	and, by Lemma \ref{L:subtofull},
	\begin{align}
	\sum \limits_{b=r_nk_n + \frac{1}{2}k_n(k_n-1)}^{h_n-1} &|\sum \limits_{i=0}^{r_{n}-k_n}\mu(T^{k_{n}(h_{n}+c_{n})}(I^{[i]}_{n,b}) \cap B) - \mu(I^{[i]}_{n,b})\mu(B)| \nonumber \\
	&= \sum \limits_{b=r_nk_n + \frac{1}{2}k_n(k_n-1)}^{h_n-1} |\frac{1}{r_n+1}\sum \limits_{i=0}^{r_{n}-k_n}\mu(T^{-ik_{n} + \frac{1}{2}k_{n}(k_{n}-1)}(I_{n,b}) \cap B) - \mu(I_{n,b})\mu(B)| \nonumber \\
	&\leq \int  \Big{|} \frac{1}{r_n+1} \sum \limits_{i=0}^{r_{n}-k_n} \chi_{B} \circ T^{-k_ni-\frac{1}{2}k_n(k_n-1)}-\mu(B)\Big{|}d\mu \nonumber \to 0
	\end{align}
	by Proposition \ref{pe} as $k_{n} \leq r_{n} + 1$.  	Since $k_{n} \leq r_{n}$, $r_{n}k_{n} + \frac{1}{2}k_{n}(k_{n}-1) \leq 2r_{n}^{2}$ and since $\frac{r_{n}^{2}}{h_{n}} \to 0$ by assumption, $(r_{n}k_{n} + \frac{1}{2}k_{n}(k_{n}-1))\mu(I_{n}) \to 0$. So sum $(\dagger)$ tends to zero.
	
	For the sum $(\ddagger)$: for $r_{n} - k_{n} + 2 \leq i < r_{n} + 1$ and $k_{n} \leq r_{n}$, since $\frac{r_{n}^{2}}{h_{n}} \to 0$ we have $k_{n}(h_{n}+c_{n}) + i(h_{n}+c_{n}) \geq (r_{n}+2)(h_{n}+c_{n}) = h_{n+1} + h_{n} + 2c_{n} - \frac{1}{2}r_{n}(r_{n}-1) \geq h_{n+1}$ so
	\begin{align*}
	T^{k_{n}(h_{n}+c_{n})}(I_{n,b}^{[i]}) &= T^{k_{n}(h_{n}+c_{n})}(I_{n+1,b + i(h_{n}+c_{n})+\frac{1}{2}i(i-1)}) \\
	&= T^{k_{n}(h_{n}+c_{n})+ i(h_{n}+c_{n})+\frac{1}{2}i(i-1)}(I_{n+1,b})
	= T^{h_{n+1}}(I_{n+1,b+h_{n}+2c_{n}-\frac{1}{2}r_{n}(r_{n}-1)}).
	\end{align*}
	Therefore, the sum $(\ddagger)$ satisfies
	\[
	\sum_{b=0}^{h_{n}-1} |\sum_{i=r_{n}-k_{n}+2}^{r_{n}} \mu(T^{k_n(h_n+c_n)}(I_{n,b}^{[i]}) \cap B) - \mu(I_{n,b}^{[i]})\mu(B)| 
	\leq \sum_{y=0}^{h_{n+1}-1} |\mu(T^{h_{n+1}}(I_{n+1,y}) \cap B) - \mu(I_{n+1,y})\mu(B)|
	\]
	which tends to zero as $\{h_{n}\}$ is rank-one uniform mixing.
	
	Since $(\dagger)$ and $(\ddagger)$ tend to $0$, we have that $(\star)$ tends to zero.  The same argument with $k_{n} + 1$ in place of $k_{n}$ shows that $(\star\star)$ tends to zero.  As $c_{n}\mu(I_{n}) \leq \frac{c_{n}}{h_{n}} \to 0$, this shows $\{ t_{n} \}$ is rank-one uniform mixing.
 \end{proof}

\begin{proof}[Proof of Theorem \ref{T:realmixing}]
	By Proposition \ref{P:finmeasest}, $T$ is on a finite measure space.
	Let $\{t_m\}$ be any sequence. Set $p_m$ such that $h_{p_m}+c_{p_m} \leq t_m < h_{p_m+1}+c_{p_m+1}$. Choose a subsequence $\{t_{m_j}\}$ of $\{t_m\}$ such that $p_{m_j}$ is strictly increasing. Then $\exists$ $\{q_n\}$ with $h_n+c_n \leq q < h_{n+1}+c_{n+1}$ such that $\{t_{m_j}\}$ is a subsequence of $\{q_n\}$ (take $\{q_n\}=\{t_{m_j}\} \cup \{h_n+c_n | \; n \: \text{ s.t.} \; \forall j, \: p_{m_j} \neq n \}$). Theorem \ref{mixingES} gives $\{q_n\}$ is mixing so $\{t_{m_j}\}$ is. As every $\{t_m\}$ has a mixing subsequence, $T$ is mixing.
\end{proof}

\textbf{Acknowledgements} We wish to thank the referee for helpful suggestions regarding terminology and organization of the results.

\dbibliography{Complexity}

\end{document}